\documentclass[12pt, a4paper]{amsart}

\newtheorem*{rep@theorem}{\rep@title}
\newcommand{\newreptheorem}[2]{%
\newenvironment{rep#1}[1]{%
 \def\rep@title{#2 \ref{##1}}%
 \begin{rep@theorem}}%
 {\end{rep@theorem}}}
\makeatother

\usepackage{amscd,amsmath,amssymb,amsthm,amsfonts}
\usepackage[alphabetic]{amsrefs}
\usepackage{mathrsfs}
\usepackage[shortlabels]{enumitem}
\usepackage[all]{xy}

\usepackage{thmtools}

\usepackage{xcolor}
\usepackage[hypertexnames=true,colorlinks = true, allcolors=blue,linktoc = all, pdffitwindow = false, urlbordercolor = white]{hyperref}%

\usepackage{tikz}
 
\usetikzlibrary{knots}
\usetikzlibrary{decorations.markings}
\usepackage{hyperref}
\usepackage{graphicx} 

\def\min{\operatorname{min}}

\newcommand{\IC}[0]{\mathbb{C}}

 \newcommand{\IN}[0]{\mathbb{N}}

 \newcommand{\IZ}[0]{\mathbb{Z}}

\newcommand{\CC}[0]{\mathcal{C}}

 \newcommand{\CL}[0]{\mathcal{L}}
 
 \newcommand{\CP}[0]{\mathcal{P}}

\newcommand{\CU}[0]{\mathcal{U}} 
\newcommand{\CW}[0]{\mathcal{W}}

\newcommand{\OF}[0]{\vec{F}}

\newcommand {\e}{{\epsilon}}

\newcommand {\eval}{{\rm eval}_{F_p}}

\newreptheorem{theorem}{Theorem}
\newtheorem{theorem}{Theorem}[section]
\newtheorem*{theorem*}{Theorem}
\newtheorem*{proposition*}{Proposition}
\newtheorem{proposition}[theorem]{Proposition}
\newtheorem{lemma}[theorem]{Lemma}
\newtheorem*{lemma*}{Lemma}
\newtheorem{example}[theorem]{Example}

\newtheorem{definition}[theorem]{Definition}

\newtheorem{corollary}[theorem]{Corollary}
\newtheorem{remark}[theorem]{Remark}

\newtheorem{convention}{Convention}[section] 

\numberwithin{equation}{section}

\begin{document}
\title[A central limit theorem for $F_p$]{An extension of Krishnan's central limit theorem to the Brown-Thompson groups}
\author{Valeriano Aiello} 
\address{Valeriano Aiello,
Dipartimento di Matematica, Universit\`a di Roma Sapienza, P.le Aldo Moro
5, 00185 Roma, Italy, \url{https://github.com/valerianoaiello}}
\email{valerianoaiello@gmail.com}

\begin{abstract}
We extend a central limit theorem, recently established for the Thompson group $F=F_2$ by Krishnan, to the Brown-Thompson groups $F_p$, where $p$ is any integer greater than or equal to $2$. The non-commutative probability
space considered is the group algebra $\mathbb{C}[F_p]$, equipped with the canonical trace. The random variables in question are $a_n:= (x_n + x_n^{-1})/\sqrt{2}$, where $\{x_i\}_{i\geq 0}$ represents the standard family of infinite generators. Analogously to the case of $F=F_2$, it is established that the limit distribution of $s_n = (a_0 + \ldots + a_{n-1})/\sqrt{n}$ converges to the standard normal distribution.

Furthermore, it is demonstrated that for a state corresponding to Jones's oriented subgroup $\vec{F}$, such a central limit theorem does not hold.
\end{abstract}

\maketitle

\section*{Introduction}
Thompson’s group $F$ was introduced by R. Thompson in the 1960s, along with its two sibling groups $T$ and $V$. While $F$ is a group of piecewise linear homeomorphisms of $[0,1]$ with derivatives that are powers of $2$ and a finite number of points of non-differentiability   at dyadic rationals, $T$ acts on $S^1$ and $V$ on the Cantor set. $F$ is one of the most intriguing countable discrete groups, yet it remains quite enigmatic, as its analytical properties have posed challenges to experts for decades. The question of its amenability is particularly notorious. For an excellent introduction to the fundamental properties of the Thompson groups, we refer the reader to \cite{CFP}. The Brown-Thompson group $F_p$, for $p \geq 2$, was introduced later, in the 1980s, by Brown \cite{Brown}. Its elements can still be seen as piecewise linear homeomorphisms of $[0,1]$, but the derivatives are powers of $p$, and the points of non-differentiability are of the form $a/p^k$, with $a, k \in \mathbb{Z}$. In particular, for $p=2$, it is evident that $F_2=F$. The Brown-Thompson groups admit an infinite presentation, namely
$$
\langle x_0, x_1, \ldots \; | \; x_nx_k=x_kx_{n+p-1} \quad \forall \; k<n\rangle\, .
$$

Thompson groups have been of interest to many areas, for example homotopy theory, cryptography, knot theory \cite{A2, AB1, AB2, Krushkal, Luo}, and more recently
non-commutative probability, \cite{Kostler, KKW, Kri}, to name but a  few. Our interest here is motivated by the latter topic.
We recall that
an algebraic non-commutative probability space is a pair 
$(A,\varphi)$, where $A$ is   a unital $*$-algebra over $\IC$, and$\varphi: A\to \IC$ is a unital positive linear functional (also known as state). 
The elements of $A$ are called non-commutative random variables.
  Given a sequence of non-commutative probability spaces $(A_n, \varphi_n)$, $n\in\IN$, $(A, \varphi)$, a sequence of variables $a_n\in A_n$, $a\in A$, we say that $a_n$ converges in distribution to $a$ (we write $a_n\stackrel{distr}{\to} a$) if 
 $$
\lim_{n\to \infty} \varphi_n(a_n)=\varphi(a)
$$
The Central Limit Theorem is a fundamental result in probability theory and it is no surprise that non-commutative versions do exist (see e.g. a "free" version of it in \cite{Voi} and the recent survey \cite{Skeide}).
Quite recently, Krishnan \cite{Kri} proved a Central Limit Theorem for the Thompson group $F$. 
Our main goal in this article is an extension of such theorem to all the Brown-Thompson groups.
More precisely, we investigate the non-commutative probability space consisting of the group algebra $\IC[F_p]$ equipped with the functional given by its canonical trace $\gamma: \IC[F_p]\to \IC$, defined as $\gamma(g):=\delta_{g,e}$, $g\in F_p$, $e$ is the neutral element.
 The main result of this article is the following, which for $p=2$ recovers exactly Krishnan's theorem. 
\begin{theorem}[Central Limit Theorem]\label{maintheo1}
Let $x_1$, $x_2$, \ldots be the generators of $F_p$ in its infinite presentation and set
$$
a_n:=\frac{x_n+x_n^{-1}}{\sqrt{2}} \qquad \text{ and } \qquad s_n:=\frac{ \sum_{i=0}^{n-1} a_i }{\sqrt{n}}
$$
be  self-adjoint random variables in the probability space $(\IC[F_p], \gamma)$, with $\gamma: \IC[F_p]\to \IC$ being the trace.
Then,  $s_n$ converges in distribution to $x$, where $x$ is a random variable normally distributed of variance $1$.
\end{theorem}
We briefly outline the idea of the proof. As we will see, by the definition of the trace in the group algebra, computing the limits of the moments of order $d$ of $s_n$ amounts to counting the number of words of length $d$, in the first $n$ generators and their inverses, that yield the neutral element $e$. To estimate the number of these words, we associate
a permutation with
 each word. The permutation is  obtained as follows: we re-write each word in the so-called normal form (which will have the same length as the original word), and then keep track of how each letter of the word changes position from the initial word, to the normal form. Finally, we estimate the number of these permutations for any $n$ and $d$.

Going back to the presentation of the results of this paper, we actually also investigate the non-commutative space where the algebra is $\IC[F]$, but with a different state. We make a brief digression in order to introduce it.
In 2014, Vaughan Jones initiated a research program  focusing on Thompson's groups, whose main component was a machinery to produce their actions. While for a more detailed account on this program and the results that followed we refer to the surveys \cite{Jo18, Bro0, A2}, in passing, we mention that this led to the discovery of new classes of groups  \cite{Bro1, Bro2}. In one implementation of this technique, Jones produced several unitary representations of Thompson groups, enabling the derivation of both previously existing and new representations, see \cite{Jo14, Jo19, BJ, ABC, AJ}. Among the new representations, one resulted in an intriguing subgroup, the so-called oriented subgroup $\vec{F}$ of $F$. This subgroup has proven to be significant in both group theory and knot theory. Remarkably, it can produce all oriented knots \cite{A}, akin to braid groups, and has also led to the discovery of a new maximal subgroup of infinite index, distinct from the previously known subgroups such as the stabilizers of points under the natural action of $F$ (see \cite{Sav, Sav2, SavBak}). 
Now, to any element $g$ of $F$, one can associate a special graph called its planar graph $\Gamma(g)$, and an element belongs to $\vec{F}$ precisely when the (normalized) chromatic polynomial evaluated at $2$, that is, Chr$_{\Gamma(g)}(2)/2$, is  equal to $1$. Given Krishnan's Central Limit Theorem, it is natural to ask whether a similar theorem holds for this new state defined by Chr$_{\Gamma(g)}(2)/2$. 
 It turns out that, while the “odd moments” of $s_n$ are all equal to zero and the second moment converges to $2$, the even moments of order greater than $2$ diverge.

We conclude the introduction with a few remarks on the structure of the paper. 
Section \ref{sec:1} is devoted to introducing the Thompson group $F$, the Brown-Thompson groups $F_p$, and Jones's oriented subgroup $\vec{F}$. These are the main objects of study in this article. 
Section \ref{sec:CLT} revisits the basic results on abstract reduction systems, which are employed in the same section to provide a proof of the Central Limit Theorem for $F_p$, namely Theorem \ref{maintheo1}. 
In the final part of the article, Section \ref{sec:oriented} demonstrates that the Central Limit Theorem does not apply to the canonical state associated with the oriented subgroup $\vec{F}$.

\section{Preliminaries and notation}\label{sec:1}
In this section, we review the definitions of Thompson's group $F$, the Brown-Thompson groups $F_k$, and the directed subgroup $\vec{F}$. Interested readers can find further details in \cite{CFP, B} regarding $F$, in \cite{Brown} for $F_3$, and in \cite{GS, Ren} for $\vec{F}$.

Thompson's group $F$ is the collection of all piecewise linear homeomorphisms within the unit interval $[0,1]$ that are diffentiable everywhere except at a finite set of dyadic rational numbers (i.e. numbers of the form $a/2^k$, $a, k\in\IZ$). Moreover, on the differentiable intervals, the derivatives are powers of $2$, with derivatives (where well defined) being powers of $2$. We adhere to the conventional notation: $f\cdot g(t)=g(f(t))$.

Thompson's group is characterized by the following infinite presentation:
$$
F=\langle y_0, y_1, y_2, \ldots \; | \; y_ny_k=y_ky_{n+1} \quad \forall \; k<n\rangle\, .
$$
Notably, $y_0$ and $y_1$ suffice to generate $F$.
The elements $y_0$, $y_1$, $y_2$, \ldots generate a monoid denoted by $F_+$, whose elements are said to be positive.
In passing, we mention that Thompson monoid $F_+$ and Thompson-like moinoids have been of interest in non-commutative probability, see e.g. \cite{KKW, CR, CDR}.

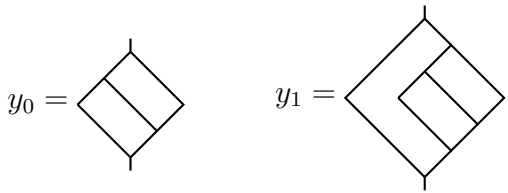
\begin{figure}
\phantom{This text will be invisible} 
\[
\begin{tikzpicture}[x=.35cm, y=.35cm,
    every edge/.style={
        draw,
      postaction={decorate,
                    decoration={markings}
                   }
        }
]

\node at (-1.5,0) {$\scalebox{1}{$y_0=$}$};
\node at (-1.25,-3) {\;};

\draw[thick] (0,0) -- (2,2)--(4,0)--(2,-2)--(0,0);
 \draw[thick] (1,1) -- (2,0)--(3,-1);

 \draw[thick] (2,2)--(2,2.5);

 \draw[thick] (2,-2)--(2,-2.5);

\end{tikzpicture}\qquad
\;\;
\begin{tikzpicture}[x=.35cm, y=.35cm,
    every edge/.style={
        draw,
      postaction={decorate,
                    decoration={markings}
                   }
        }
]

\node at (-3.5,0) {$\scalebox{1}{$y_1=$}$};
\node at (-1.25,-3.25) {\;};

\draw[thick] (2,2)--(1,3)--(-2,0)--(1,-3)--(2,-2);

\draw[thick] (0,0) -- (2,2)--(4,0)--(2,-2)--(0,0);
 \draw[thick] (1,1) -- (2,0)--(3,-1);

 \draw[thick] (1,3)--(1,3.5);
 \draw[thick] (1,-3)--(1,-3.5);

\end{tikzpicture}
\]
\caption{The generators of $F=F_2$.}\label{genThompsonF}
\end{figure}
The projection of $F$ onto its abelianization is denoted by $\pi: F\to F/[F,F]=\mathbb{Z}\oplus \mathbb{Z}$, and it takes on a meaningful interpretation when viewing $F$ as a group of homeomorphisms: $\pi(f)=(\log_2 f'(0),\log_2 f'(1))$. Notably, $\pi(y_0)=(1,-1)$, and $\pi(y_k)=(0, -1)$ for all $k\geq 1$.
An alternative description pertinent to this paper emerges: the elements of $F$ can be visualized as pairs $(T_+,T_-)$ of planar binary rooted trees, each having the same number of leaves. When represented visually, one tree is positioned upside down on top of the other; $T_+$ constitutes the top tree, while $T_-$ constitutes the bottom tree.

Such a representation is termed a tree diagram, and any pair of trees $(T_+,T_-)$ depicted in this manner is considered a tree diagram. Two pairs of trees are deemed equivalent if they differ solely by pairs of opposing carets. In other words, two tree diagrams are equivalent if they can be transformed into one another by exchanging pairs of carets positioned opposite to each other.
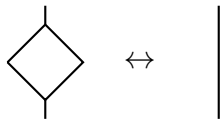
\[\begin{tikzpicture}[x=.5cm, y=.5cm,
    every edge/.style={
        draw,
      postaction={decorate,
                    decoration={markings}
                   }
        }
]

 \draw[thick] (0,0)--(1,1)--(2,0)--(1,-1)--(0,0); 
 \draw[thick] (1,1.5)--(1,1); 
 \draw[thick] (1,-1.5)--(1,-1); 
\node at (0,-1.2) {$\;$};
\node at (3.5,0) {$\scalebox{1}{$\leftrightarrow$}$};

\end{tikzpicture}
\begin{tikzpicture}[x=.5cm, y=.5cm,
    every edge/.style={
        draw,
      postaction={decorate,
                    decoration={markings}
                   }
        }
]

  \draw[thick] (1,1.5)--(1,-1.5); 
\node at (0,-1.2) {$\;$};
 
\end{tikzpicture}
\]
Each equivalence class of pairs of trees, representing an element of $F$, corresponds uniquely to a reduced tree diagram. Here, "reduced" means that the number of vertices in the diagram is minimized, as outlined in \cite{B}. Figure \ref{genThompsonF} provides a depiction of the reduced tree diagrams for the first two generators of $F$.

Moving on, there exists a family of groups that extends the scope of the Thompson group, these are the Brown-Thompson groups. For any integer $p\geq 2$, the Brown-Thompson group $F_p$ can be defined by the following presentation:
$$
\langle x_0, x_1, \ldots \; | \; x_nx_k=x_kx_{n+p-1} \quad \forall \; k<n\rangle\, .
$$
The elements $x_0, x_1, \ldots , x_{p-1}$ generate $F_p$.
The neutral element is denoted by $e$.

 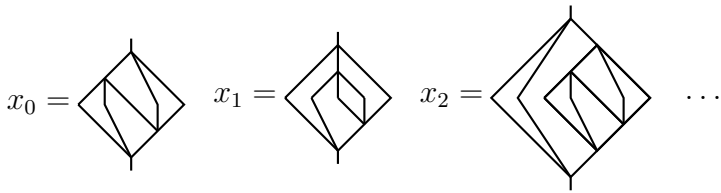
\begin{figure}
\phantom{This text will be invisible} \[
\begin{tikzpicture}[x=.35cm, y=.35cm,
    every edge/.style={
        draw,
      postaction={decorate,
                    decoration={markings}
                   }
        }
]

\node at (-1.5,0) {$\scalebox{1}{$x_0=$}$};
\node at (-1.25,-3) {\;};

\draw[thick] (0,0) -- (2,2)--(4,0)--(2,-2)--(0,0);
\draw[thick] (1,1) -- (1,0)--(2,-2);
\draw[thick] (1,1) -- (2,0)--(3,-1);
\draw[thick] (2,2) -- (3,0)--(3,-1);

 \draw[thick] (2,2)--(2,2.5);
 \draw[thick] (2,-2)--(2,-2.5);

\end{tikzpicture}
\;\;
\begin{tikzpicture}[x=.35cm, y=.35cm,
    every edge/.style={
        draw,
      postaction={decorate,
                    decoration={markings}
                   }
        }
]

\node at (-1.5,0) {$\scalebox{1}{$x_1=$}$};
\node at (-1.25,-3.25) {\;};


\draw[thick] (0,0) -- (2,2)--(4,0)--(2,-2)--(0,0);
 \draw[thick] (1,0)--(2,-2);
\draw[thick] (3,0)--(2,1) -- (1,0); 
\draw[thick] (2,0)--(3,-1);
\draw[thick] (2,2) -- (2,0);
\draw[thick] (3,0)--(3,-1);


 \draw[thick] (2,2)--(2,2.5);
 \draw[thick] (2,-2)--(2,-2.5);
\end{tikzpicture}
\;\;
\begin{tikzpicture}[x=.35cm, y=.35cm,
    every edge/.style={
        draw,
      postaction={decorate,
                    decoration={markings}
                   }
        }
]

\node at (-3.5,0) {$\scalebox{1}{$x_2=$}$};
\node at (-1.25,-3.25) {\;};

\draw[thick] (2,2)--(1,3)--(-2,0)--(1,-3)--(2,-2);

\draw[thick] (0,0) -- (2,2)--(4,0)--(2,-2)--(0,0);
 \draw[thick] (1,1) -- (2,0)--(3,-1);

\draw[thick] (0,0) -- (2,2)--(4,0)--(2,-2)--(0,0);
\draw[thick] (1,1) -- (1,0)--(2,-2);
\draw[thick] (1,1) -- (2,0)--(3,-1);
\draw[thick] (2,2) -- (3,0)--(3,-1);

\draw[thick] (1,3) -- (-1,0)--(1,-3);

\node at (6,0) {$\ldots$};

 \draw[thick] (1,3)--(1,3.5);
 \draw[thick] (1,-3)--(1,-3.5);

\end{tikzpicture}
\]
\caption{The generators of the Brown-Thompson group $F_3$.}\label{genThompsonF3}
\end{figure}

Likewise to $F$, the elements of $F_p$ find their description through pairs of $p$-ary trees. An illustration of the generators of $F_3$ is presented in Figure \ref{genThompsonF3}.

\begin{convention}\label{conventiondrawings}A convention is established for the visualization of trees on the plane. In our representation of planar   trees, the roots are depicted as vertices of degree $3$. Consequently, each tree diagram features the topmost and bottommost vertices of degree $1$, positioned respectively on the lines $y=1$ and $y=-1$. The leaves of the trees are situated on the $x$-axis, precisely at the non-negative integers. As for the elements of $F_3$, 
the leaves are positioned on the $x$-axis, precisely at the non-negative integers $\mathbb{N}_0 :=\{0, 1, 2, \ldots \}$. 
\end{convention}

A natural embedding $\iota: F_2\to F_3$ is established by substituting the $3$-valent vertices in a tree diagram of $F_2$ with $4$-valent vertices and connecting the middle edges in the only possible planar manner (see Figure \ref{basic-tree}). In this embedding, we observe that $\iota(y_i)=x_{2i}$. For simplicity, we will frequently omit the symbol for this mapping in our notation.

Let $(T_+,T_-)$ represent an element of $F_3$. We now provide a brief recapitulation, derived from \cite{Jo14, Jo18}, of how to construct $\Gamma(T_+,T_-)$, referred to as the \emph{planar graph of $(T_+,T_-)$}. 

Each tree diagram $(T_+,T_-)$ divides the strip bounded by two lines parallel to the $x$-axis, 
one passing through the highest univalent vertex, 
one through the lowest univalent vertex. We color the strip alternately in black and white. The leftmost region is conventionally black. The vertices of the planar graph are positioned on the $x$-axis, specifically at $-1/2+2\mathbb{N}_0 :=\{-1/2,1+1/2, 3+1/2, \ldots \}$, with one vertex corresponding to each black region. An edge is drawn between two black regions whenever they meet at a $4$-valent vertex . 

For an element $(T_+,T_-)$ in $F_2$, we define the \emph{planar graph of $(T_+,T_-)$} to be that of $\iota(T_+,T_-)\in F_3$.

Note that the  
planar graph of $(T_+,T_-)$ is
essentially the Tait graph of $\CL(T_+,T_-)$ 
except that we do not specify the signs of its edges. 

The oriented subgroups were introduced by Jones, the first in \cite{Jo14}, the second in \cite{Jo18}. 
We are  
now in a position to define them
\begin{align*}
\vec{F}&:=\{(T_+,T_-)\in F \; | \; \Gamma(T_+,T_-) \textrm{ is $2$-colourable} \} \\  
&:=\{(T_+,T_-)\in F \; | \; {\rm Chr}_{\Gamma(T_+,T_-)}(2)=2\} \\  
\vec{F}_3&:=\{(T_+,T_-)\in F_3\; | \; \Gamma(T_+,T_-) \textrm{ is $2$-colourable} \}\\
&:=\{(T_+,T_-)\in F_3 \; | \; {\rm Chr}_{\Gamma(T_+,T_-)}(2)=2\} 
\end{align*}
where ${\rm Chr}_{\Gamma}(t)$ is the chromatic polynomial, and by being $2$-colourable we mean that it is possible to label the vertices of the graph with two colours such that whenever two vertices are connected by an edge, they have different colours.  
We designate the two colors as $+$ and $-$. Importantly, the definition of the planar graph $\Gamma(T_+,T_-)$ is independent of the specific representative $(T_+,T_-)$, as indicated in \cite[Section 4.1]{Jo14}.
Since $\Gamma(T_+,T_-)$ is connected, if it is $2$-colorable, there exist precisely two colorings. By convention, we opt for the coloring where the left-most vertex has the color $+$.  Notably, for $\Gamma(T_+,T_-)$, being $2$-colorable is equivalent to being bipartite.

Here we summarize some results on $\vec{F}$.
\begin{theorem} \label{theoGS}
\cite{GS}
The oriented subgroup  $\vec{F}$
\begin{enumerate}
\item  is   isomorphic to the Brown-Thompson group $F_3$ via the isomorphism $\alpha:  F_3 \to F$ defined by mapping $x_0$, $x_1$, $x_2\in F_3$ to 
$w_0:=y_0y_1$,
$w_1:=y_1y_2$,
$w_2:=y_2y_3\in F$. More generally, we have $\alpha(x_i)=y_iy_{i+1}$ for all $i\geq 0$. In particular, $\vec{F}$ is contained in $K_{(1,2)}$. 
\item is the stabilizer of the subset of dyadic rationals 
$$
S:=\{ .a_1\ldots a_n \, | \, \omega(a_1\ldots a_n )\equiv_2 0\}
$$ 
where $\omega: W_2\to \IZ_2$ is defined as 
$\omega(a_1\ldots a_n ):=\sum_{i=1}^n a_i$, $W_2$ is the set of binary words (i.e. with letters in $\{0, 1\}$), $\equiv_2$ denotes equivalence modulo $2$.
\end{enumerate}
\end{theorem}
\begin{remark}
In \cite{Ren},  Ren   gave a graphical interpretation
   of the isomorphism 
  $\alpha: F_3\to \vec{F}$  of Theorem \ref{theoGS}-(1): on the level of tree diagrams, it  is
  obtained by
   replacing every $4$-valent vertex of ternary trees by the \emph{basic} binary tree with $4$ leaves shown in Figure \ref{basic-tree}), see also for other isomorphisms of this type
\cite{BCS}.
\end{remark}
\begin{remark}
The introduction of $\vec{F}$ (or more precisely its image under a monomorphism) 
provided a novel type of maximal subgroup of infinite index  in $F$. 
This contributed to the development of a research area devoted to the study of maximal subgroup of infinite index in Thompson groups, \cite{GS2, G2, TV,  TV2, BBQS, TV3}.
\end{remark}

\begin{figure}
\phantom{This text will be invisible} 
 \[
\begin{tikzpicture}[x=1cm, y=1cm,
    every edge/.style={
        draw,
      postaction={decorate,
                    decoration={markings}
                   }
        }
]

 \draw[thick] (1,1)--(1,1.5);
\draw[thick] (0,0) -- (1,1)--(2,0);
 \draw[thick] (1,0)--(1,1);

\end{tikzpicture}
\quad
\begin{tikzpicture}[x=1cm, y=1cm,
    every edge/.style={
        draw,
      postaction={decorate,
                    decoration={markings}
                   }
        }
]

\node at (-.75,0.25) {$\scalebox{1}{$\mapsto$}$};

 \draw[thick] (1,1)--(1,1.5);
\draw[thick] (0,0) -- (1,1)--(2,0);
 \draw[thick] (1,0)--(.5,.5);

\end{tikzpicture}
\]
\caption{Ren's map: a ternary tree is replaced by the \emph{basic tree}.}\label{basic-tree}
\end{figure}
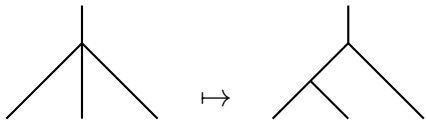
 
Other subgroups that will appear in this article are the rectangular subgroups of $F$, as introduced in \cite{BW}.
\begin{align*}
K_{(a,b)}&:=\{f\in F\; | \; \log_2f'(0)\in a\IZ, \log_2f'(1)\in b\mathbb{Z}\}=\pi^{-1}(a\IZ \oplus b\IZ) \quad a, b\in\IN.
\end{align*}
These subgroups can be characterised as the only finite index subgroups of $F$ isomorphic with $F$ \cite[Theorem 1.1]{BW}.

The following lemma can be proved either by means of Theorem \ref{theoGS}-(1), or by drawing the corresponding planar graphs.
\begin{lemma}\label{operationsonE}
Let  
$\varphi_R, \varphi_L : F\to  F$ be the right and left shift endomorphisms, which are   defined graphically  as
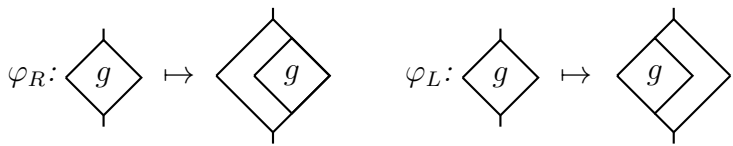
\[
\begin{tikzpicture}[x=1cm, y=1cm,
    every edge/.style={
        draw,
      postaction={decorate,
                    decoration={markings}
                   }
        }
]

\node at (-.45,0) {$\scalebox{1}{$\varphi_R$:}$};

\draw[thick] (0,0)--(.5,.5)--(1,0)--(.5,-.5)--(0,0);
\node at (1.5,0) {$\scalebox{1}{$\mapsto$}$};

\node at (.5,0) {$\scalebox{1}{$g$}$};
\node at (.5,.-.75) {$\scalebox{1}{}$};

 \draw[thick] (.5,.65)--(.5,.5);
 \draw[thick] (.5,-.65)--(.5,-.5);
 
\end{tikzpicture}
\begin{tikzpicture}[x=1cm, y=1cm,
    every edge/.style={
        draw,
      postaction={decorate,
                    decoration={markings}
                   }
        }
]
 
\node at (2.5,0) {$\scalebox{1}{$g$}$};

\draw[thick] (1.5,0)--(2.25,.75)--(3,0)--(2.25,-.75)--(1.5,0);
\draw[thick] (2,0)--(2.5,.5)--(3,0)--(2.5,-.5)--(2,0);
 
 \node at (1.5,.-.75) {$\scalebox{1}{}$};

 \draw[thick] (2.25,.75)--(2.25,.9);
 \draw[thick] (2.25,-.75)--(2.25,-.9);

\end{tikzpicture}
\qquad 
\begin{tikzpicture}[x=1cm, y=1cm,
    every edge/.style={
        draw,
      postaction={decorate,
                    decoration={markings}
                   }
        }
]

\node at (-.45,0) {$\scalebox{1}{$\varphi_L$:}$};

\draw[thick] (0,0)--(.5,.5)--(1,0)--(.5,-.5)--(0,0);
\node at (1.5,0) {$\scalebox{1}{$\mapsto$}$};

\node at (.5,0) {$\scalebox{1}{$g$}$};
\node at (.5,.-.75) {$\scalebox{1}{}$};

 \draw[thick] (.5,.65)--(.5,.5);
 \draw[thick] (.5,-.65)--(.5,-.5);
 
\end{tikzpicture}
\begin{tikzpicture}[x=1cm, y=1cm,
    every edge/.style={
        draw,
      postaction={decorate,
                    decoration={markings}
                   }
        }
]
 
\node at (2.5,0) {$\scalebox{1}{$g$}$};

\draw[thick] (3.5,0)--(2.75,.75)--(2.5,0.5);
\draw[thick] (2.5,-0.5)--(2.75,-.75)--(3.5,0);
\draw[thick] (2,0)--(2.5,.5)--(3,0)--(2.5,-.5)--(2,0);
 
 \node at (1.95,.-.75) {$\scalebox{1}{}$};


 \draw[thick] (2.75,.75)--(2.75,.9);
 \draw[thick] (2.75,-.75)--(2.75,-.9);

\end{tikzpicture}
\]
Note that $\varphi_R(y_i)=y_{i+1}$ for all $i\geq 0$.
When $F$ is seen as a group of homeomorphisms of the unit interval, for $g\in F$, these 
maps read as follow 
\begin{align*}
&\varphi_L(g)(t):=\left\{\begin{array}{ll}
t & \text{ if } t\in [0,1/2]\\
g(2t-1) & \text{ if } t\in [1/2 ,1]
\end{array}
\right.\\
&\varphi_R(g)(t):=\left\{\begin{array}{ll}
g(2t) & \text{ if } t\in [0,1/2 ]\\
t & \text{ if } t\in [1/2,1]\\
\end{array}
\right.
\end{align*}
Then 
\begin{enumerate}
\item $g\in\OF$ if and only if $\varphi_R(g)\in\OF$;
\item $g\in\OF$ if and only if $\varphi_L(g)\in\OF$.
 \end{enumerate}
\end{lemma}

\section{A central limit for $F_p$ w.r.t. the tracial state}\label{sec:CLT}

The goal of this section is to prove a central limit theorem for $F_p$. 
Here the probability space consists of the group algebra $\IC[F_p]$ and its tracial state $\gamma: \IC[F_p]\to \IC$ defined as $\gamma(g):=\delta_{g,e}$.

We begin by introducing the notations and recalling some basic results on abstract reduction systems. For the latter topic, we refer to \cite{BadNip}.
\begin{definition}
For $d\in\IN$, $S_d$  denotes the group of permutations of the set $[d] := \{1, \ldots , d\}$.
For any even $d$, a pair-partition $\pi$ of $[d]$ is collection of $d/2$ of disjoint subsets $V_1$, \ldots ,  $V_{d/2}$ such that
$\cup_{i=0}^{d/2}V_i=[d]$. 
There is a distinguished pair-partition called the rainbow  pair-partition $\pi_{rain} =\{ \{1, d\}, \{2, d-1\}, \{d/2, (d+2)/2\}\}$. 
When $d$ is odd, there are no pair-partitions.
$\CP_2(d)$  denotes the set of pair partitions of $[d]$. 
\end{definition}

\begin{definition}
The generators of $F_p$ (i.e., $x_0$, $x_1$, \ldots) are called letters. 
Given $d\in\IN$, a word of length $d$ in $F_p$ is a tuple $w=(x_{i(1)}^{\epsilon(1)}, \ldots, x_{i(d)}^{\epsilon(d)})$, where $i: [d]\to \IN_0$ and 
$\epsilon: [d]\to\{1,-1\}$.  A word is also denoted by $(i,\epsilon)$. The evaluation of a word $w=(x_{i(1)}^{\epsilon(1)}, \ldots, x_{i(d)}^{\epsilon(d)})$ is $\eval(w)=x_{i(1)}^{\epsilon(1)} \cdots x_{i(d)}^{\epsilon(d)}$. If the evaluation of a word is the neutral element of the group $e$, we say that the word is neutral. 
The set of words of length $d$ are denoted by $\CW(d)$.
Any two words $w_1=(x_{i(1)}^{\epsilon(1)}, \ldots, x_{i(d)}^{\epsilon(d)})$ and $w_2=(x_{i'(1)}^{\epsilon'(1)}, \ldots, x_{i'(d')}^{\epsilon'(d')})$ can be concatenated, that is $w_1w_2=(x_{i(1)}^{\epsilon(1)}, \ldots, x_{i(d)}^{\epsilon(d)}, x_{i'(1)}^{\epsilon'(1)}, \ldots, x_{i'(d')}^{\epsilon'(d')})$. A word $w$ is a subword of $w'$ if $w'=w_1ww_2$, with $w_1$, $w_2$ possibly being empty words.
The set of neutral word of length $d$ is 
$$
\CW_0(d):=\{w\in \CW(d) \, | \, \eval(w)=e\}.
$$
The set of neutral word of length $d$ in the first $n$ generators is 
$$
\CW_0(d, n):=\{w\in \CW(d) \, | \, \eval(w)=e, i: [d]\to \{0, \ldots, n-1\}\}.
$$
\end{definition}

\begin{definition}
An abstract reduction system is defined as a pair $(A, \to)$, where $A$ is a non-empty set and $\to$ is a binary relation, represented as a subset of $A \times A$. We denote $(x, y) \in \to$ by $x \to y$. 
\begin{enumerate}
\item Given a binary relation $R \subseteq A \times B$ and $S \subseteq B \times C$ be two binary relations. Their composition is defined by 
$S\circ R:=\{(x,z)\in A\times C\, | \, \exists y\in B \text{ with } (x,y)\in R, (y, z)\in S\}$.
\item For any $i \in \mathbb{N}$, the $i$-fold composition of $\to$ is denoted as $\stackrel{i}{\to}$, while the reflexive and transitive closure of $\to$ is represented by $\stackrel{*}{\to}$.
\item An element $x \in A$ is termed reducible if there exists some $y \in A$ such that $x \to y$; otherwise, it is referred to as irreducible or in normal form. A normal form of $x$ is an element $y$ such that $x \stackrel{*}{\to} y$ and $y$ is in normal form.
\item Two elements $x$ and $y$ in $A$ are deemed joinable if there exists some $z \in A$ such that $x \stackrel{*}{\to} z$ and $y \stackrel{*}{\to} z$, denoted as $x \downarrow y$.
\item A reduction is considered confluent if for all $w$, $x$, and $y$ in $A$ with $w \stackrel{*}{\to} x$ and $w \stackrel{*}{\to} y$, we have $x \downarrow y$. \item It is termed locally confluent if for all $w$, $x$, and $y$ in $A$ with $w \to x$ and $w \to y$, we have $x \downarrow y$.
\item A reduction is defined as terminating if there exists no infinite chain $x_0 \to x_1 \to \ldots$
\item A reduction is considered normalizing if every element has a normal form.
\item A reduction $\to$ is termed finitely branching if every element has finitely many direct successors.
\end{enumerate}
\end{definition}
The following lemma, especially point (4), will be our main tool in this section.
\begin{lemma} \label{abs-red-sys}
\begin{enumerate}
\item If a reduction is normalizing and confluent, then every element has a unique normal form.
\item A finitely branching reduction system
 terminates if and only if there is a monotone embedding into $(\IN,  >)$.
\item (Newman's Lemma) A terminating reduction is confluent if and only if it is locally confluent.
\item \cite[Prop. 2.9]{Kri} Let $(A, \rightsquigarrow)$ be an abstract reduction system with a locally confluent and terminating reduction and $A^{irr}:=\{w\in A\, | \, w $ is irreducible w.r.t. $ \rightsquigarrow\}$. Let $(A^{irr}, \rightarrowtail)$ another abstract reduction system with a locally confluent and terminating reduction. Then every element admits a normal form with respect to $\to :=  \stackrel{*}{\rightarrowtail}\circ \stackrel{*}{\rightsquigarrow}$. 
\end{enumerate}
\end{lemma}

The normal form that we will most use in this article is associated with the composition
of the following binary relations
\begin{enumerate}
\item $(\CW(d), \rightsquigarrow)$ where $w_1ww_2 \rightsquigarrow w_1f(w)w_2$ for $w_1$, $w_2\in\CW(d)$, $w\in\CW(2)$ and
\begin{align*}
f(w)&=f(x_{i(k)}^{-1}, x_{i(k+1)}):=\left\{\begin{array}{ll}
(x_{i(k+1)}, x_{i(k)+p-1}^{-1}) & i(k)>i(k+1)\\
(x_{i(k+1)+p-1}, x_{i(k)}^{-1}) & i(k)<i(k+1)\\
(x_{i(k+1)}, x_{i(k)}^{-1}) & i(k)=i(k+1)\\
\end{array}
\right.
\end{align*}
\item $(\CW(d)^{irr}, \rightarrowtail)$ where $w_1ww_2 \rightarrowtail w_1h(w)w_2$ for $w_1$, $w_2\in\CW(d)$, $w\in\CW(2)$ and
\begin{align*}
h(w)&=h(x_{i(k)}^{\epsilon(k)}, x_{i(k+1)}^{\epsilon(k+1)})\\
&:=\left\{\begin{array}{ll}
(x_{i(k+1)-p+1}, x_{i(k)}) & \text{ if } i(k+1)-p+1>i(k), \epsilon(k)=\epsilon(k+1)=1\\
(x_{i(k+1)}^{-1}, x_{i(k)-p+1}^{-1}) & \text{ if } i(k)-p+1>i(k+1), \epsilon(k)=\epsilon(k+1)=-1\\
\end{array}
\right.
\end{align*}
\end{enumerate}
The abstract reduction system $(\CW(d),\to)$, where $\to$ is defined as $\stackrel{*}{\rightarrowtail}\circ \stackrel{*}{\rightsquigarrow}$, is terminating and confluent (the proof consists of a very small modification of the arguments of  \cite[Prop. 2.14]{Kri} and, therefore, is omitted). By Lemma \ref{abs-red-sys}-(4), we have the following result.
\begin{proposition}\label{prop2.14}
For any $w=(i,\epsilon)\in\CW(d)$, there exists a unique word in $w'=(j,\epsilon')= (x_{j(1)}, \ldots,  x_{j(r)},  x_{j(r+1)}^{-1}, \ldots,  x_{j(d)}^{-1}) \in  \CW(d)$ such that
\begin{enumerate}
\item $r\leq d$ 
\item $\eval (w)= \eval (w')$
\item $j(1)+p-1\geq j(2)$,  $j(2)+p-1\geq j(3)$, \ldots,  $j(r-1)+p-1\geq j(r)$,
\item $j(r+1)\leq j(r+2)+p-1$,  $j(r+2)\leq j(r+3)+p-1$, \ldots,  $j(d-1)\leq j(d) +p-1$,  
\end{enumerate}
\end{proposition}
For $w$ as in Proposition \ref{prop2.14}, we refer to $(x_{j(1)}, \ldots,  x_{j(r)})$ and $(x_{j(r+1)}^{-1}, \ldots,  x_{j(d)}^{-1})$ as its positive and negative parts, respectively.

From now on, if we refer to a normal form without specifying the reduction system, we mean that of Proposition \ref{prop2.14}.
In the case of $F_2=F$, this normal form 
 is called anti-normal form by Belk in \cite{B}.

We now associate a permutation with any word. We will make use of these permutations to compute the $\gamma(s_n^d)$.
\begin{definition}
Given a word $w=(i, \epsilon)\in\CW(d)$, denote by $(j, \epsilon')$ the unique normal form of $w$ (defined in Proposition \ref{prop2.14}).
For each $l\in [d]$, the letter $x_{i(l)}^{\epsilon(l)}$ is transformed to some $x_{j(u_l)}^{\epsilon(l)}$. We define a permutation $\tau\in S_d$ as
\begin{align}
\tau(i,\epsilon)(l)&:=u_l \text{ for any } l\in [d]
\end{align}
To any word $(i, \e)\in \CW_0(d)$ we assign a pair partition $\pi(i, \e):=\tau(i,\epsilon)(\pi_{rain})$ in $\CP_2(d)$.
\end{definition}

Denote by $(1_k)$ the array of length $k$ whose entries are all equal to $1$.
Here follows a basic result on the permutations associated with words of length $d$.
\begin{lemma}\label{lemma3.4}
Given $w=(i, \epsilon)\in\CW(d)$ and its normal form $(j, \epsilon')$. Then the permutation $\tau(i,\epsilon)$ satisfies the following properties
\begin{enumerate}
\item $-d(p-1)\leq j(l)-i(\tau(i,\epsilon)^{-1}(l))\leq d(p-1)$ for $l\in [d]=\{1, \ldots, d\}$;
\item if $\sum_{l=1}^d\epsilon(l)=0$, then $d$ is even and $\epsilon'=(1_{d/2}, -1_{d/2})$ and  
$$-d(p-1)/2\leq j(l)-i(\tau(i,\epsilon)(l))\leq d(p-1)/2 \qquad\forall l\in [d]=\{1, \ldots, d\}.$$
\end{enumerate}
\end{lemma}
\begin{proof}
(1) In a word of length $d$ each letter can change  position at most $d$ times and each time the index changes by $p-1$.

(2) First we consider the upper bound. In this case, there are $d/2$ letters with exponent $1$, $d/2$ with exponent $-1$. 
If we look at the binary relations $\rightsquigarrow$ and $\rightarrowtail$ we notice that with latter the indices either stay the same or decrease, so we only need to consider the first.
Each letter of the form $x_l^{-1}$ can move at most $d/2$ times to the right of a letter of the form $x_h$.  
With the binary relation $\rightsquigarrow$, the index either increase by $p-1$, or stays the same (symmetrically, each  letter of the form $x_h$
moving to the left of a letter of type  $x_l^{-1}$, either increase by $p-1$, or stays the same). 
This means that at the end the index increases at most by $d(p-1)/2$.

As for the lower bound,  the binary relation $\rightarrowtail$ is the only one not to increase the indices. 
Each letter of the form $x_l$ can move at most $d/2$ times to the right of a letter of the form $x_h$ and, thus, at the end the
index decreases at most by $-d(p-1)/2$.
The same reasoning applies for pairs of letters of the form $x_l^{-1}$ and $x_h^{-1}$.
\end{proof}

Here follows a simple, yet useful result on the pair partition $\pi(i, \e)$.
\begin{lemma}\label{lemma3.8}
Let $(i,\e)\in\CW_0(d)$ with $i_0=\min\{i(l)\, |\, l\in [d]\}$, then for each pair $\{l, m\}\in \pi(i,\e)$ it holds $i(l)=i_0$ if and only if $i(m)=i_0$.
\end{lemma}
\begin{proof}
In the reduction system $(\CW(d), \to)$, every time two letters are exchanged, the smallest index remains unchanged. Therefore, in order to have
a letter to cancel out an occurrence of $x_{i_0}^{\pm 1}$, it must be of the form $x_{i_0}^{\mp 1}$.
\end{proof}

\begin{remark} \label{abs_red_tech}
Here we introduce an instrumental abstract reduction system $(\CW(d), \rightharpoonup)$ that we will make use of in a couple of technical lemmas. 
Given a minimal a word $w$, take the minimal index occurring in it, that is $i_0=\min\{i(l)\, |\, l\in [d]\}$. We now define an abstract reduction system $(\CW(d), \rightharpoonup)$ that "pushes" the letters with index $i_0$ to the sides of $w$, more precisely all occurrences of $x_0$ are moved to left, those of $x_0^{-1}$ to the right.
Given a word $u=w_1ww_2$, with $w$ of length $2$,  we set 
$$
u=w_1ww_2\rightharpoonup w_1f(w)w_2
$$
where
\[
\begin{array}{lll}
f(x_{i(m)}^{\e(m)}, x_{i_0})=(x_{i_0}, x_{i(m)+p-1}^{\e(m)}) & \text{ for } i(m)>i_0\\
f(x_{i_0}^{-1}, x_{i(m)}^{\e(m)})=(x_{i(m)+p-1}^{\e(m)}, x_{i_0}^{-1}) & \text{ for } i(m)>i_0\\
f(x_{i_0}^{-1}, x_{i_0})=(x_{i_0}, x_{i_0}^{-1}) &
\end{array}
\]
The reduction system is terminating because there are at most $d$ occurrences of $x_{i_0}^\pm$ and each one is moved at most $d-1$ times.

As for the confluence, whenever the two reductions take place in two disjoint sub-words,  then the words are joinable simply taking the word obtained by 
applying the two commuting reductions one after the other. 
We have to take care of the case where we apply the reduction on two overlapping sub-words (of length $2$)
inside a word of length $3$.
\[
\begin{array}{cccccc}
 & &  x_{i(m)+p-1}^{\e(m)}x_{i_0}^{-1} x_{i_0} & \rotatebox{-45}{$\stackrel{2}{\rightharpoonup}$} &  &\\
 &  \rotatebox{45}{$\rightharpoonup$} & & & & \\
x_{i_0}^{-1} x_{i(m)}^{\e(m)}x_{i_0} & & & & & x_{i_0} x_{i(m)+2p-2}^{\e(m)}x_{i_0}^{-1}\\
 &  \rotatebox{-45}{$\rightharpoonup$} & & & & \\
 &  & x_{i_0}^{-1} x_{i_0} x_{i(m)+p-1}^{\e(m)} &   \rotatebox{45}{$\stackrel{2}{\rightharpoonup}$} & &  \\
\end{array}
\]
By Newman's Lemma, there exists a unique normal form. Therefore, every element $w$ can be rewritten as $x_{i_0}^r \tilde{w} x_{i_0}^{-r'}$ for some $r, r'\in \IN_0$ and some word $\tilde{w}=(\tilde{i}, \tilde{\e})$. 
\end{remark}

The next results provides a description of the normal form of a neutral word and a property enjoyed by its indices.
\begin{corollary}\label{corollary3.5}
Given $w=(i, \epsilon)\in\CW_0(d)$ and its normal form $(j, \epsilon')$. Then 
\begin{enumerate}
\item there exists a normal form 
$$
(j,\epsilon)=(x_{j(1)}, \ldots,  x_{j(d/2)},  x_{j(d/2)}^{-1}, \ldots,  x_{j(1)}^{-1})
$$
with $j(1)+p-1\geq j(2)$, \ldots , $j((d/2) -1)+p-1\geq j(d/2)$ and 
$$
-d(p-1)/2\leq j(l) - i(\tau(i,\epsilon)^{-1}(l))\leq d(p-1)/2 \qquad l\in \{1, \ldots, d/2\}.
$$
\item for any $k\in [d/2]=\{1, \ldots, d/2\}$ it holds
$$
|i(\tau(i,\epsilon)^{-1}(k)) - i(\tau(i,\epsilon)^{-1}(d-k+1))|\leq d(p-1)\, .
$$
\end{enumerate}
\end{corollary}
\begin{proof}
(1) Note that the integer $d$ is necessarily even, so it is of the form $d=2k$ for some $k\in\IN$.
By Proposition \ref{prop2.14} we have that any word $w$ may be written in normal form, that is as $w=(x_{j(1)}, \ldots,  x_{j(d/2)},  x_{j(d/2+1)}^{-1}, \ldots,  x_{j(d)}^{-1})$. We need to show that $j(l)=j(d-l+1)$ for $l\in [d/2]$. This is done by induction on $k$. For $k=1$, that is $d=2$, 
 all words in $\CW_0(2)$ clearly satisfy the claim. 

Suppose that $k\geq 1$ and that there are at least two different generators occurring in the normal form (distinct up to taking inverses), otherwise the claim is clear.
Use the abstract reduction system $(\CW(d), \rightharpoonup)$ of Remark \ref{abs_red_tech} to re-write $w$ in its normal form $w'$. 
Let $i_0=\min\{i(l)\, |\, l\in [d]\}$ be the minimal index. We have $w'=x_{i_0}^r\tilde w x_{i_0}^{-r'}$ for some $r, r'\in\IN$.
The letters of $w$ and $w'$ whose index is  minimal remain unchanged when applying the defining relations. As $\eval(w)=\eval(w')=e$, it follows that $r'=r$. Therefore, $w=e$ if and only if $\tilde w=e$. The word $\tilde w$ has length strictly shorter than $w$ and by inductive hypothesis we have
$$
\tilde w = (x_{\tilde i(1)}, \ldots,  x_{\tilde i(d'/2)},  x_{\tilde i(d'/2)}^{-1}, \ldots,  x_{\tilde i(1)}^{-1})
$$
for some $d'<d$.
As $w=x_{i_0}^r\tilde w x_{i_0}^{-r}$, it follows that 
$$
w=(x_{j(1)}, \ldots,  x_{j(d/2)},  x_{j(d/2)}^{-1}, \ldots,  x_{j(1)}^{-1})
$$
(2) 
By part (1) 
 it holds 
$$
-d(p-1)/2\leq j(l) - i(\tau(i,\epsilon)^{-1}(l))\leq d(p-1)/2 \qquad l\in \{1, \ldots, d/2\} ,
$$
then by making use of the triangle inequality and of $ j(d-k+1)=j(k)$, we get 
\begin{align*}
& |i(\tau(i,\epsilon)^{-1}(k)) - i(\tau(i,\epsilon)^{-1}(d-k+1)\\
&\leq |i(\tau(i,\epsilon)^{-1}(k)) - j(d-k+1)| +  |j(d-k+1)-i(\tau(i,\epsilon)^{-1}(k))|\\
&=|i(\tau(i,\epsilon)^{-1}(k)) - j(d-k+1)| +  |j(d-k+1)-i(\tau(i,\epsilon)^{-1}(k))|\\
&\leq d(p-1)/2 + d(p-1)/2 =d(p-1) 
\end{align*}
 \end{proof}

The next two proposition will show that a partially known word can be "completed" in a unique way in order to be neutral.
\begin{proposition}\label{prop3.9}
Let $d\in\IN$ and $(i, \e)\in \CW_0(d)$, such that $\e(l_+)=1$, $\e(l_-)=- 1$  for all $\{l_+, l_-\}\in \pi(i, \e)$.
If all values of $i(l_+)$ are known, those of $i(l_-)$ are uniquely determined.
\end{proposition}
\begin{proof}
The integer $d$ is necessarily even, thus $d=2k$.
We give a proof by induction on $k$. If $k=1$, there is only one pair partition $\{1, 2\}$ and the claim is the content of Lemma \ref{lemma3.8}.

Suppose that the thesis holds for $d=2k\geq 2$ and let us prove that it holds for $d'=d+2=2(k+1)$. 
Here we make use   of the abstract reduction system of Remark \ref{abs_red_tech}.
Every element $w$ can be rewritten as $x_{i_0}^r \tilde{w} x_{i_0}^{-r'}$ for some  $r, r'\in\IN_0$ and some word $\tilde{w}=(\tilde{i}, \tilde{\e})$. For each pair of $\pi(i, \e)$, if $l_+=i_0$, then $l_-=i_0$ by Lemma \ref{lemma3.8} and, thus, $r=r'$. 
Note that the evaluation $\eval(\tilde{w})$ is equal to $e$. 
Now the pairs of $\pi(\tilde i, \tilde \e)$ are uniquely determined by the inductive hypothesis and from this we may recover the values of $l_-$ for the pairs of $\pi(i, \e)$.
\end{proof}

\begin{proposition}\label{prop3.11}
Let $d\in 2\IN$, $\tau\in S_d$, $\e: [d]\to \{1, -1\}$ such that 
\begin{align*}
\e(\tau^{-1}(l))&=\left\{\begin{array}{ll}
1 & \text{ for } l\in [d/2]=\{1, \ldots, d/2\} \\
-1 & \text{ for }  l\in \{d/2+1, \ldots, d\} 
\end{array}
\right.
\end{align*}
Suppose 
\begin{align}\label{starid}
i(\tau^{-1}(l)) < i(\tau^{-1}(l-1))-(4p-4)d\qquad \text{ for } l\in \{1, \ldots, d/2\}
\end{align}
 then there exists a unique $i(\tau^{-1}(l))$ for $l\in \{d/2+1, \ldots, d\}$ such that $(i, \e)\in\CW_0(d)$
and $\tau(i, \e)=\tau$.
\end{proposition}
\begin{proof}
The strategy of the proof consists of showing that, even though the word is only partially known, 
we may find the expression of the normal form explicitly 
 because
the way  the binary relation $\to$ changes the indices is completely determined by our assumptions. Finally, by knowing that 
the original word is neutral, we will determine all the unknown indices. 

Our goal is to find that there exists a unique word $w=(i,\e)\in \CW_0(d)$ with $\tau(i, \e)=\tau$.   
According to Corollary \ref{corollary3.5}-(2), the following is a necessary condition
\begin{align*}
&|i(\tau(i,\epsilon)^{-1}(k)) - i(\tau(i,\epsilon)^{-1}(d-k+1))|\leq d(p-1)\, \qquad \forall k\in [d/2].
\end{align*}
Hence, we restrict to the words satisfying this property
\begin{align}\label{squareid}
 i(\tau^{-1}(k))-d(p-1) & \leq i(\tau^{-1}(d-k+1))\leq  i(\tau^{-1}(k))+d(p-1)
\end{align}
for all $k\in [d/2]$.

We now use the abstract reduction system $(\CW(d), \to)$ on $w$.
Even though we do not know explicitly $w$, we first show that the normal form can be computed.
Each $x_{ i(\tau^{-1}(l))}^{\e(l)}$ is transformed to $x_{ j(v_l)}^{\e(l)}$, with $j(v_l)=i(\tau^{-1}(l))+q(l)$, $|q(l)|\leq d(p-1)/2$.
The process of computing the normal form consists of two parts: in the first we \emph{move} the generators $\{x_i\}_{i\geq 0}$ to the left, their inverses $\{x_i^{-1}\}_{i\geq 0}$  to the right with the binary relation $\rightsquigarrow$; 
in the second we re-arrange the elements in the positive part of the word and those in the negative part with  the binary relation $\rightarrowtail$.

At each step of the first part of the process, we have a subword of length $2$ of the form
$$
(x_{ i(\tau^{-1}(d-m+1))+r}^{-1}, x_{ i(\tau^{-1}(k))+q})
$$
with $m, k\in [d/2]$, $|q|, |r|\leq d(p-1)/2$, $r$ and $q$ being known (at the very first step $r$ and $q$ are both $0$).

\noindent
\textbf{Case $m>k$.}
\begin{align*}
i(\tau^{-1}(d-m+1))+r &	\stackrel{\eqref{squareid}}{\leq} i(\tau^{-1}(m))+d(p-1)+r\\
&	\stackrel{\eqref{starid}}{<}i(\tau^{-1}(k))-(4p-4)d+d(p-1)+r\\
&=i(\tau^{-1}(k))-(3p-3)d+r\\
&\leq i(\tau^{-1}(k))+q
\end{align*}
then 
\begin{align*}
&(x_{ i(\tau^{-1}(d-m+1))+r}^{-1}, x_{ i(\tau^{-1}(k))+q})\\
&=(x_{ i(\tau^{-1}(k))+q+p-1}, x_{ i(\tau^{-1}(d-m+1))+r}^{-1})
\end{align*}

\noindent
\textbf{Case $m<k$.}
 We have
\begin{align*}
i(\tau^{-1}(d-m+1))+ r &\stackrel{\eqref{squareid}}{\geq}  i(\tau^{-1}(m))+r-d(p-1)\\
&\stackrel{\eqref{starid}}{>} i(\tau^{-1}(k))+(4p-4)d-d(p-1)+r\\
&=i(\tau^{-1}(k))+(3p-3)d+r\\
&\geq i(\tau^{-1}(k))+q
\end{align*}
Then
\begin{align*}
&(x_{ i(\tau^{-1}(d-m+1))+r}^{-1}, x_{ i(\tau^{-1}(k))+q})\\
&=(x_{ i(\tau^{-1}(k))+q}, x_{ i(\tau^{-1}(d-m+1))+r+p-1}^{-1})
\end{align*}

\noindent
\textbf{Case $m=k$.}
Whenever $m=k$ and the two elements $x_{ i(\tau^{-1}(d-m+1))+r}^{-1}$ and $x_{ i(\tau^{-1}(k))+q}$ appear consecutively, they must have the same index (i.e. $i(\tau^{-1}(k))+q = i(\tau^{-1}(d-m+1))+r$)
\begin{align*}
&(x_{ i(\tau^{-1}(d-m+1))+r}^{-1}, x_{ i(\tau^{-1}(k))+q})\\
&=(x_{ i(\tau^{-1}(k))+q}, x_{ i(\tau^{-1}(d-m+1))+r}^{-1})
\end{align*}

As for the second part of the process yielding the normal form we have to consider two cases.\\
\noindent
\textbf{Case $m>k$.}
We have a subword of length $2$ of the form 
$$
(x_{i(\tau^{-1}(m)) +r}, x_{i(\tau^{-1}(k)) +q})
$$
Since
\begin{align*}
i(\tau^{-1}(m))+r &\stackrel{\eqref{starid}}{<}i(\tau^{-1}(k) )-4d(p-1)+r\\
&\leq i(\tau^{-1}(k) )-4d(p-1)+d(p-1)/2\\
&=i(\tau^{-1}(k) )-7d(p-1)/2\\
&<i(\tau^{-1}(k) )+q -p+1\\
\end{align*}
we have
$$
(x_{i(\tau^{-1}(m)) +r}, x_{i(\tau^{-1}(k)) +q})\rightarrowtail (x_{i(\tau^{-1}(k)) +q-p+1}, x_{i(\tau^{-1}(m)) +r})
$$

\noindent
\textbf{Case $m<k$.}

We have a subword of length $2$ of the form 
$$
(x_{i(\tau^{-1}(d-m+1) +r}^{-1}, x_{i(\tau^{-1}(d-k+1) +q}^{-1})
$$
Since
\begin{align*}
i(\tau^{-1}(d-k+1))+q&\stackrel{\eqref{squareid}}{\leq} i(\tau^{-1}(k)) +d(p-1)+q \\
&\stackrel{\eqref{starid}}{<}i(\tau^{-1}(m))-4d(p-1)+d(p-1) +q\\
&\leq i(\tau^{-1}(m))-4d(p-1)+d(p-1) +d(p-1)2^{-1} \\
&= i(\tau^{-1}(m))-5d(p-1)2^{-1}\\
&=i(\tau^{-1}(m))-2d(p-1)-d(p-1)2^{-1}\\
&\leq i(\tau^{-1}(m))-2d(p-1)+r\\
&\stackrel{\eqref{squareid}}{<} i(\tau^{-1}(d-m+1))+ r-p+1
\end{align*}
we have
$$
(x_{i(\tau^{-1}(d-m+1)) +r}^{-1}, x_{i(\tau^{-1}(d-k+1)) +q}^{-1})\rightarrowtail (x_{i(\tau^{-1}(d-k+1)) +q}, x_{i(\tau^{-1}(d-m+1)) +r-p+1})
$$

Now we show that the of values of $j(v_l)$ are actually equal to $l$. Indeed, for $l\in [d/2]$ we have
 we have
\begin{align*}
j(v_l)&= i(\tau^{-1}(l))+q(l)\stackrel{\eqref{starid} }{<} i(\tau^{-1}(l-1))+q(l)-(4p-4)d\\
&\leq i(\tau^{-1}(l-1))+d(p-1)2^{-1}-(4p-4)d\\
&= i(\tau^{-1}(l-1))+d\left( \frac{-7p+7}{2}\right)\\ 
&< i(\tau^{-1}(l-1))+q(l-1)\\
&< i(\tau^{-1}(l-1))+q(l-1)+p-1=j(v_{l-1})+p-1\\
\end{align*}
and, thus, by the uniqueness of the normal form, we have $l=v_l$ for all $l\in [d/2]$.

Since $w=(i,\e)\in \CW_0(d)$ is a normal form, 
we must have $v_{d-k+1}=d-v_{k}+1$ for all $k\in [d/2]$. 
It follows that $v_{d-k+1}=d-k+1$ for all $k\in [d/2]$ and we are done.  
\end{proof}
\begin{example}
\begin{enumerate}
\item An application of Proposition \ref{prop3.9}. 
Consider the word $(i, \e)=(x_{i_1}^{-1}, x_{i_2}^{-1}, x_{i_3}^{-1}, x_3, x_0, x_1)$, $\e=(-1, -1, -1, 1, 1, 1)$,
$\pi(i, \e) = \{ \{ 3, 5\}, \{ 2, 6\}, \{ 1, 4\} \}$. This pair partition is telling us, that after re-ordering the letters, $x_{i_2}^{-1}$ will cancel out $x_1$, 
$x_{i_1}^{-1}$ will cancel out $x_3$, $x_{i_3}^{-1}$ will cancel out $x_0$. Being $0$ the minimal index, we have that $i_3=0$.
Using  the abstract reduction system $(\CW(d), \rightharpoonup)$ of Remark \ref{abs_red_tech} we re-write $(i, \e)$ as
\begin{align*}
&(x_0, x_{i_1+p-1}^{-1}, x_{i_2+p-1}^{-1},  x_{3+2p-2}, x_{1+p-1}, x_{0}^{-1})\\
&=(x_0, x_{i_1+p-1}^{-1}, x_{i_2+p-1}^{-1},  x_{2p+1}, x_{p}, x_{0}^{-1})
\end{align*}
which leads us to the neutral word $(x_{i_1+p-1}^{-1}, x_{i_2+p-1}^{-1},  x_{2p+1}, x_{p})$, with pair partition $\{\{ 1, 3\}, \{ 2, 4\}\}$.
Now the minimal index is $p$ and thus $i_2+p-1=p$, that is $i_2=1$. As before, we re-write the new word and obtain
$(x_{p}, x_{i_1+2p-2}^{-1},  x_{4p-1}, x_{p}^{-1})$, which implies $i_1=2p+1$.
\item As for Proposition \ref{prop3.9}, given a pair partition and partly filled word, it is not always possible to complete it in such a way that it belongs to $\CW_0(d)$. One might consider the word $(x_{i(1)}^{-1}, x_{i(2)}^{-1}, x_{i(3)}^{-1}, x_3, x_0, x_1)$ and pair partition $\{\{ 1, 5\}, \{ 2, 4\}, \{ 3, 6\}\}$, with $p>2$. 
Using the reduction system $(\CW(d), \rightharpoonup)$ of Remark \ref{abs_red_tech} we are led to $i(3)=-p+2$ which is impossible.
\item An application of Proposition \ref{prop3.11}. Let $p$ be equal to $5$ and consider $\tau=(1,3,2,6,5)$, $\e=(1,-1,1,-1,1,-1)$, and the initial word is $(x_1, x_{i(2)}^{-1}, x_{50}, x_{i(4)}^{-1}, x_{100}, x_{i(6)}^{-1})$. By following the steps of the proof of Proposition \ref{prop3.11} we get that the normal form is
$$
(x_{i(\tau^{-1}(1))-4}, x_{i(\tau^{-1}(2))-4)}^{-1}, x_{i(\tau^{-1}(3))}, x_{i(\tau^{-1}(4))}^{-1}, x_{i(\tau^{-1}(5))}^{-1}, x_{i(\tau^{-1}(6))}^{-1})
$$
and as the corresponding evaluation must be a neutral word, we get $i(\tau^{-1}(1))-4=i(\tau^{-1}(6))$, 
 $i(\tau^{-1}(2))-4=i(\tau^{-1}(5))$,
  $i(\tau^{-1}(3))=i(\tau^{-1}(4))$. We know that $i(\tau^{-1}(1))=100$, $i(\tau^{-1}(2))=50$, $i(\tau^{-1}(3))=1$. Then, we have 
  $i(\tau^{-1}(4))=1$, $i(\tau^{-1}(5))=46$, $i(\tau^{-1}(6))=100$ and the word is
  $(x_1, x_{100}^{-1}, x_{50}, x_{1}^{-1}, x_{100}, x_{46}^{-1})$.
\end{enumerate}
\end{example}

We are finally approaching the proof of the Central Limit Theorem. We only need a couple of inequalities that will provide lower and upper bounds for $\gamma(s_n^d)$.
For $n, d\in\IN$, $\tau\in S_d$, we set
\begin{align}
\CW_0(d, n, \tau)&:=\{	(i,\e)\in\CW_0(d, n) \, | \, \tau(i, \e)=\tau	\}\\
N(d, n, \tau)& := |\CW_0(d, n, \tau)|
\end{align}
 
 \begin{lemma}\label{prop4.3}
For $d\in 2\IN$ and $n\in\IN$, then 
$$
N(d, n, \tau) \leq \binom{n-d(d+1)(p-1)}{d/2}
$$
\end{lemma} 
\begin{proof}
Set $\CU(d,n)$ as
$$
\{k: [d/2]\to	\{0,\ldots, n-1\}\, | \, k(l+1)\leq k(l)+(d+1)(p-1) \quad \forall l\in [d/2-1] \}
$$
first we want to show that the map 
\begin{align}
&\iota: \CW_0(d, n, \tau)\to \CU(d,n)\\
&\iota((i, \e))(l):=i(\tau^{-1}(l))
\end{align} 
is injective. 

Let us check that $\iota((i, \e))(k+1)\leq \iota((i, \e))(k)+d(p-1)+p-1$ for all $l\in [d/2-1]$.
Denote by $(j, \e')$ the unique normal form of $(i, \e)$.
By Corollary \ref{corollary3.5}-(1) we have
\begin{align*}
i(\tau^{-1}(l+1))& \leq j(l+1) +\frac{d(p-1)}{2}\leq j(l)+p-1+\frac{d(p-1)}{2} \\
&\leq i(\tau^{-1}(l))+d(p-1) +p-1
\end{align*}

The map $\iota$ is injective. Indeed, suppose that for $(i, \e), (i',\e')\in\CW(d,n,\tau)$ we have $\iota((i, \e))=\iota((i', \e)')$, that is 
$i(\tau^{-1}(l))=i'(\tau^{-1}(l))$ for all $l\in [d/2-1]$. 

As any normal form $(i, \e)\in\CW_0(d,n,\tau)$ satisfies $\e(\tau^{-1}(l))=1$ for all $l\in [d/2]$, $\e(\tau^{-1}(l))=-1$ for all $l\in [d]\setminus [d/2]$, we have $\e=\e'$. 
We want to show that  $i(\tau^{-1}(l))=i'(\tau^{-1}(l))$ for all $l\in [d/2]$ implies 
$i(\tau^{-1}(l))=i'(\tau^{-1}(l))$ for all $l\in \{d/2+1, \ldots, d\}$. This is the content of Proposition \ref{prop3.9}.

It remains to calculate of the cardinality of $\CU(d,n)$. This is done by means of an injection of   $\CU(d,n)$ into the set of function
$$
\{k: [d/2]\to \{-d(d+1)(p-1), \ldots, n-1\} \, | \, k(l+1)<k(l) \quad \forall k\in [d/2-1]\}.
$$
Given $k\in \CU(d, n)$ we map it to the function $\tilde k(m):= k(m)-2m(d+1)(p-1)$. Indeed, 
\begin{align*}
\tilde k(m+1)&= k(m+1)-2(m+1)(d+1)(p-1)\\
&\leq k(m)+(d+1)(p-1)-2(m+1)(d+1)(p-1)\\
&\leq (k(m)+2m(d+1)(p-1))+ (d+1)(p-1)-2(d+1)(p-1)\\
&= \tilde k(m)-(d+1)(p-1)\\
& < \tilde k(m)
\end{align*}
It is now easy to see that the cardinality of this set is
$$
\binom{n-d(d+1)(p-1)}{d/2} \, .
$$
\end{proof}

\begin{lemma}\label{prop4.4}
For $d\in 2\IN$ and $n\in\IN$, $n>d(p-1)+d(4p-4)+\frac{d^2}{2}(4p-4)+\frac{d}{2}$, then 
$$
\binom{ n-d(p-1)-d(4p-4)-\frac{d^2}{2}(4p-4)}{d/2}  \leq N(d, n, \tau)
$$
\end{lemma} 
\begin{proof}
Let $c, r\in\IN$, we set $\CL(c,r)$
$$
\{k:[c/2]\to \{0, \ldots, r-1\}\, | \, k(l) < k(l-1)-(4p-4)c\quad \forall l\in \{2, \ldots, c/2\}\}
$$
We fix $\tau\in S_d$ and $k\in \CL(d, n-d(p-1))$, we want to construct a word in $\CW_0(d, n, \tau)$. First we set 
$i(\tau^{-1}(l))=k(l)$ for $l\in [d/2]$, 
\[
\epsilon(\tau^{-1}(l))=\left\{\begin{array}{ll}
1 & \text{ for } l\in [d/2]\\
-1 & \text{ otherwise}
\end{array}\right.
\] 
By definition $i$ satisfies $i(\tau^{-1}(l))< i(\tau^{-1}(l-1))-(4p-4)d$ for all $l\in \{2, \ldots,d/2\}$ and, thus, by Proposition \ref{prop3.11} the map $i$
can be extended in a unique way to $[d]$ and in such a way that  $(i, \e)\in \CW_0(d,n)$ and $\tau(i, \e)=\tau$.
By Corollary \ref{corollary3.5}-(2),  the extension of $i$ to $[d]$ takes values in $\{0, \ldots, n-1\}$. Hence,   $(i, \e)\in \CW_0(d,n,\tau)$. 
We have an injection of $\CL(d, n-d(p-1))$ into $\CW_0(d, n, \tau)$.

Now if we define an injection of the set 
\begin{align*}
\CC(d)&:=\{k: [d/2]\to \{  d(4p-4)+\frac{d^2}{2}(4p-4), \ldots, n-d(p-1)-1\} \\
&\qquad\qquad \, | \, k(l+1)<k(l) \quad \forall l\in [d/2]\}
\end{align*}
into $\CL(d, n-d(p-1))$. Once we have it, we are done because
$$
|\CC(d)|= \binom{ n-d(p-1)-d(4p-4)-\frac{d^2}{2}(4p-4)}{d/2}  
$$
The injection from $\CC(d)$ to $\CL(d, n-d(p-1))$ takes a function $k$ and sends it to $\tilde k(l):=k(l) -d(l+1)(4p-4)$. Indeed, we have 
\begin{align*}
\tilde k(l+1)&=k(l+1) -d(l+2)(4p-4)\\
&<k(l) -d(l+1)(4p-4)-d(4p-4) \\
&= \tilde k(l) -d(4p-4)  
\end{align*}
If $k[d/2]\subseteq\{-d(4p-4)+\frac{d^2}{2}(4p-4), \ldots, n-d(p-1)-1\}$, then 
$\tilde k[d/2]\subseteq\{0, \ldots, n-d(p-1)-1\}$. 

\end{proof}

\begin{lemma}\label{limitlemma}
For $d\in 2\IN$, $n\in\IN$, we have 
$$
 \lim_{n\to\infty} \frac{1}{(2n)^{d/2}}N(d, n, \tau) = \frac{1}{(d/2)!}
 $$
\end{lemma}
\begin{proof}
Thanks to Lemmas \ref{prop4.3} and \ref{prop4.4} we have 
\begin{align*}
\frac{1}{(d/2)!} \leq \lim_{n\to\infty} \frac{1}{(2n)^{d/2}}N(d, n, \tau) \leq \frac{1}{(d/2)!}\; .
\end{align*}
\end{proof}

Before stating and proving the main result of this article, we recall this lemma.

\begin{lemma}\label{lemma4.1}
\cite[Lemma 4.1]{Kri} Let $\pi\in\CP_2(d)$. Then there exists $d!!$ permutations $\tau\in S_d$ such that $\tau(\pi)=\pi_{rain}$.
\end{lemma}

\begin{reptheorem}{maintheo1}[Central Limit Theorem] Let $x_1$, $x_2$, \ldots be the generators of $F_p$ in its infinite presentation and set
$$
a_n:=\frac{x_n+x_n^{-1}}{\sqrt{2}} \qquad \text{ and } \qquad s_n:=\frac{\sum_{i=0}^{n-1} a_i}{\sqrt{n}}
$$
be  self-adjoint random variables in the probability space $(\IC[F_p], \gamma)$, with $\gamma: \IC[F_p]\to \IC$ being the trace.
Then,  $s_n$ converges in distribution to $x$, where $x$ is a random variable normally distributed with variance 1.
\end{reptheorem}
\begin{proof}
It suffices to prove that 
$$
\lim_{n\to\infty} \gamma(s_n^d)=\left\{
\begin{array}{ll}
(d-1)!! & \text{ if  } d\in 2\IN\\
0 & \text{ otherwise}
\end{array}
\right.
$$
By Lemma \ref{limitlemma} we have
\begin{align*}
\lim_{n\to\infty} \gamma(s_n^d)&= \lim_{n\to\infty}\frac{1}{(2n)^{d/2}} \sum_{\stackrel{i: [d]\to \{0, \ldots, n-1\},} {\e: [d]\to \{1, -1\}}} \gamma(x_{i(1)}^{\e(1)}\cdots x_{i(d)}^{\e(d)})\\
&= \lim_{n\to\infty}\frac{1}{(2n)^{d/2}} \sum_{\stackrel{i: [d]\to \{0, \ldots, n-1\},} {\stackrel{ \e: [d]\to \{1, -1\}} {\eval(i, \e)=e} }} \gamma(x_{i(1)}^{\e(1)}\cdots x_{i(d)}^{\e(d)}) \\
&= \lim_{n\to\infty} \frac{1}{(2n)^{d/2}} |\CW_0(d,n)|
\end{align*}
For $d$ odd, the set $\CW_0(d,n)$ is empty and we reached the desired formula. 
We may continue the proof assuming that $d$ is even
\begin{align*}
\lim_{n\to\infty} \gamma(s_n^d)&= \lim_{n\to\infty} \frac{1}{(2n)^{d/2}} \sum_{\pi\in \CP_2(d)} \sum_{\stackrel{\tau \in S_d}{\tau(\pi)=\pi_{rain}}}|\CW_0(d,n, \tau)|\\
&= \lim_{n\to\infty} \frac{1}{(2n)^{d/2}} \sum_{\pi\in \CP_2(d)} \sum_{\stackrel{\tau \in S_d}{\tau(\pi)=\pi_{rain}}} N(d,n, \tau)\\
&= \sum_{\pi\in \CP_2(d)} \frac{1}{(2)^{d/2}}   \sum_{\stackrel{\tau \in S_d}{\tau(\pi)=\pi_{rain}}}  \lim_{n\to\infty} \frac{1}{(n)^{d/2}} N(d,n, \tau)
\end{align*}
\begin{align*}
&= \sum_{\pi\in \CP_2(d)} \frac{1}{(2)^{d/2}}   \sum_{\stackrel{\tau \in S_d}{\tau(\pi)=\pi_{rain}}} \frac{1}{(d/2)!} \\
&\stackrel{\text{Lemma \ref{lemma4.1}}}{=} \sum_{\pi\in \CP_2(d)} 1 = (d-1)!!
\end{align*}
As the limits of these are those of the normal distribution $N(0, 1)$, we are done thanks to \cite[Theorem 30.2]{Bil}.  
\end{proof}

 \section{On the Chromatic polynomial and $\vec{F}$}\label{sec:oriented}
 In this section, we consider another functional $\theta: \IC[F]\to \IC$, which arises in the framework of Jones's representations of $F$ and show that a Central Limit Theorem does not hold 
 because some of the sequences $\theta(s_n^d)$ diverge (see Proposition \ref{lastlemma}). 
 As before, we   examine the random variables
$$
a_n:=\frac{x_n+x_n^{-1}}{\sqrt{2}} \qquad s_n:=\frac{\sum_{i=0}^{n-1} a_i}{\sqrt{n}}
$$
and consider the positive linear functional $\theta: \mathbb{C}[F] \to \mathbb{C}$ defined as
 $$
 \theta(g):=\frac{{\rm Chr}_{\Gamma(g)}(2)}{2} \quad \text{ for any $g\in F$.}
 $$
 When evaluated at an element of $F$, this functional returns $1$ if the element belongs to $\OF$, and $0$ otherwise.
We aim to compute the limits of $\theta(s_n^d)$ for all $d\in\IN$.

\begin{table}[ht]
\centering
\begin{tabular}{|c|c|c|c|c|c|c|c|c|c|}
\hline 
& n=1 & n= 2& n=3 &n=4 & n=5 &n=6 &n=7 &n=8 &n=9 \\      
\hline 
d=2 &2 & 6 & 10 & 14 & 18 & 22 &  26 &  30 & 34 \\
\hline
d=4 &6 & 52 & 176 & 394 & 708 & 1118 & 1624 & & \\
\hline
d=6 &20 & 506 & 3594 & 13442 & & & & & \\
\hline
d=8 &70 & 5240 & 78962 & & & & & & \\
\hline
d=10 &252 & & & & & & & & \\ 
\hline
\end{tabular}
\caption{The table displays the values of un-normalized moments $c_n^d:=(\sqrt{2 n})^d\theta(s_n^d)$ for some values of $n$, $d$ in $\mathbb{N}_0$.
The code used to compute them may be found at \url{https://github.com/valerianoaiello/CLT_Thompson}}
\label{tab:mytable}
\end{table}

The following lemma immediately follows from Theorem \ref{theoGS}.
\begin{lemma}
The equality $c_n^d=(\sqrt{2 n})^d\theta(s_n^d)=0$ holds for all $d\in 2\IZ + 1$.
\end{lemma}
\begin{proof}
The power $s_n^d$ consists of words of length $d$ in the generators of $F$.
As $\pi(x_0)=(1,-1)$, $\pi(x_k)=(0, -1)$ for all $k\geq 1$,  we have that $\pi$ of each of these terms is in $\IZ \oplus  (2\IZ + 1)$. 
However, by Theorem \ref{theoGS}, we know that $\pi(\OF)\subset \IZ \oplus 2\IZ$. 
The statement follows.
\end{proof}
As the first row of the Table \ref{tab:mytable} suggests, there is a simple formula for the un-normalized moments of order $2$.
\begin{lemma}
The equality $c_n^2=(2 n)\theta(s_n^2)=4n-2$ holds for all $n\in \IZ$.
In particular, 
$$
\lim_{n\to \infty} \frac{\theta(s_n^2)}{2n}=2
$$
\end{lemma}
\begin{proof}
The claim of this lemma is a consequence of the following fact: the elements $x_i^{m}x_j^{n}$ (with $m, n\in\{1,-1\}$) belong to $\vec{F}$ if and only if either $j=i+1$ and $m=n=1$, or $i=j+1$, $m=n=-1$.
This can be proved in the following way: first by using that $g\in \vec{F}$ if and only if $g^{-1}\in \vec{F}$, we may assume that $i\leq j$;
then by using Lemma \ref{operationsonE} we may assume that  $i$ is zero; 
finally a direct inspection of the third leaf in the reduced tree diagram of $x_i^{m}x_j^{n}$,
reveals that in the third interval of the standard dyadic partition described by the top tree, the points in $S$ are not preserved 
(the number of right edges in the paths from the leaf to the two roots of the tree diagram have
indeed different parity) and thus by Theorem \ref{theoGS} we have proved the claim.

Now we have that
\begin{align*}
(2 n)\theta(s_n^2)&=\theta\left(\left( \sum_{i=0}^{n-1} x_i +x_i^{-1}\right)^2\right)\\
&=\theta\left(	\sum_{i,j=0}^{n-1} x_ix_j + \sum_{i,j=0}^{n-1}x_i^{-1}x_j^{-1}+ \sum_{i,j=0}^{n-1}x_ix_j^{-1}+\sum_{i,j=0}^{n-1}x_i^{-1}x_j	\right)\\ 	
&=\theta\left(	\sum_{i,j=0}^{n-1} x_ix_j\right) + \theta\left(\sum_{i,j=0}^{n-1}x_i^{-1}x_j^{-1}\right)+ \theta\left(\sum_{i,j=0}^{n-1}x_ix_j^{-1}\right)+\theta\left(\sum_{i,j=0}^{n-1}x_i^{-1}x_j	\right)\\ 	
\end{align*}
As $\theta(g)=\theta(g^{-1})$, the first two terms are equal, and so do the last two. We have
\begin{align*}
(2 n)\theta(s_n^2) &=2 \theta\left(	\sum_{i,j=0}^{n-1} x_ix_j\right) + 2\theta\left(\sum_{i,j=0}^{n-1}x_ix_j^{-1}\right)\\ 	
&=2 \theta\left(	\sum_{\stackrel{i, j=0,} {i\neq j}}^{n-1} x_ix_j\right) + 2\theta\left(\sum_{\stackrel{i, j=0,} {i\neq j}}^{n-1}x_ix_j^{-1}\right)+2n\\ 	
&=2 (n-1) + 2(n-1)+2n=4n-2\; . 	
\end{align*}

\end{proof}
The next result shows that a Central Limit Theorem cannot hold in this context.
\begin{proposition}\label{lastlemma}
For $d\geq 2$,
the limit equality 
$$
\lim_{n\to \infty} \theta(s_n^{2d})=\infty
$$
holds.
\end{proposition}
\begin{proof}
By Theorem \ref{theoGS}-(1), for each word $(j, \e)$ (here we consider words with letters being the generators of $F_3$ and their inverses), with $j: [d-1]\to \{0, \ldots,  n-1\}$, $\e': [d-1]\to \{1,-1\}$ 
we have an element $\alpha({\rm eval}_{F_3}(j,\e'))\in \vec{F}$, which is described by a new word $(\tilde j, \tilde \e')$,
$\tilde j: [2d-2]\to \{0, \ldots,  n\}$, $\tilde \e': [2d-2]\to \{1,-1\}$,
 with letters being the generators of $F$. 
 The word $(\tilde j, \tilde \e')$ is obtained by replacing each occurrence of the letter $x_j$ with the word $(y_{i}, y_{i+1})$ and each occurrence of the letter 
 $x_j^{-1}$ with $(y_{i+1}^{-1}, y_i^{-1})$.
  
 For each of these latter words, we may insert a word of length $2$ of the form $(x_i, x_i^{-1})$ or $(x_i^{-1},x_i)$ 
 at the beginning, at the end, and between any two consecutive letters. 
 Therefore, for any such a word $(j,\e)$ 
 in $F_3$ of length $d-1$,
  we obtain $(2n)^{2d-1}$ words in $F$ (of length $2d$) such that $\theta$ takes value $1$ on them.
Therefore,  we have
\begin{align*}
\lim_{n\to \infty}\theta(s_{n+1}^{2d})&= \lim_{n\to \infty}\sum_{ \stackrel{i: [2d]\to \{0, \ldots,  n\}}{\e: [2d]\to \{1,-1\}} } \frac{\theta({\rm eval}_F(i,\e))}{(2n+2)^d}\\
&\geq \lim_{n\to \infty}(2n)^{2d-1} \sum_{ \stackrel{j: [d-1]\to \{0, \ldots,  n-1\}}{\e': [d-1]\to \{1,-1\}} } \frac{\theta(\alpha({\rm eval}_{F_3}(j,\e')))}{(2n+2)^d}\\
&=\lim_{n\to \infty} (2n)^{2d-1} \sum_{ \stackrel{j: [d-1]\to \{0, \ldots,  n-1\}}{\e': [d-1]\to \{1,-1\}} } \frac{1}{(2n+2)^d}\\
&= \lim_{n\to \infty} (2n)^{2d-1}  \frac{(2n)^{d-1}}{(2n+2)^d}=\infty
\end{align*}

\end{proof}

\section*{Acknowledgements}
We would like to thank the referee for their thorough review of our manuscript, and for the comments that led to significant improvements in the presentation.

\end{document}